\documentclass[11pt]{article}
\usepackage{amssymb}
\usepackage{amsmath}
\usepackage{graphicx}
\usepackage{setspace}

\begin{document}

\newtheorem{theorem}{Theorem}[section]
\newtheorem{tha}{Theorem}
\newtheorem{conjecture}[theorem]{Conjecture}
\newtheorem{corollary}[theorem]{Corollary}
\newtheorem{lemma}[theorem]{Lemma}
\newtheorem{claim}[theorem]{Claim}
\newtheorem{proposition}[theorem]{Proposition}
\newtheorem{construction}[theorem]{Construction}
\newtheorem{definition}[theorem]{Definition}
\newtheorem{question}[theorem]{Question}
\newtheorem{problem}[theorem]{Problem}
\newtheorem{remark}[theorem]{Remark}
\newtheorem{observation}[theorem]{Observation}

\newcommand{\ex}{{\mathrm{ex}}}
\newcommand{\Ker}{{\mathrm{Ker}}}

\newcommand{\EX}{{\mathrm{EX}}}

\newcommand{\AR}{{\mathrm{AR}}}

\def\endproofbox{\hskip 1.3em\hfill\rule{6pt}{6pt}}
\newenvironment{proof}%
{%
\noindent{\it Proof.}
}%
{%
 \quad\hfill\endproofbox\vspace*{2ex}
}
\def\qed{\hskip 1.3em\hfill\rule{6pt}{6pt}}
\def\ce#1{\lceil #1 \rceil}
\def\fl#1{\lfloor #1 \rfloor}
\def\lr{\longrightarrow}
\def\e{\varepsilon}
\def\cA{{\cal A}}
\def\cB{{\cal B}}
\def\cC{{\cal C}}
    \def\cD{{\cal D}}
\def\cE{{\cal E}}
\def\cF{{\cal F}}
\def\cG{{\cal G}}
\def\cH{{\cal H}}
\def\ck{{\cal K}}
\def\cI{{\cal I}}
\def\cJ{{\cal J}}
\def\cL{{\cal L}}
\def\cM{{\cal M}}
\def\cP{{\cal P}}
\def\cQ{{\cal Q}}
\def\cS{{\cal S}}
\def\cT{{\cal T}}
\def\cW{{\cal W}}

\def\bkl{{\cal B}^{(k)}_\ell}
\def\cmkt{{\cal M}^{(k)}_{t+1}}
\def\cpkl{{\cal P}^{(k)}_\ell}
\def\cckl{{\cal C}^{(k)}_\ell}

\def\pkl{\mathbb{P}^{(k)}_\ell}
\def\ckl{\mathbb{C}^{(k)}_\ell}
\def\cthreel{\mathbb{C}^{(3)}_\ell}
\def\codd{\mathbb{C}^{(k)}_{2t+1}}
\def\ceven{\mathbb{C}^{(k)}_{2t+2}}

\def\imp{\Longrightarrow}
\def\1e{\frac{1}{\e}\log \frac{1}{\e}}
\def\ne{n^{\e}}
\def\rad{ {\rm \, rad}}
\def\equ{\Longleftrightarrow}
\def\pkl{\mathbb{P}^{(k)}_\ell}

\def\c2ell{\mathbb{C}^{(2)}_\ell}
\def\codd{\mathbb{C}^{(k)}_{2t+1}}
\def\ceven{\mathbb{C}^{(k)}_{2t+2}}
\def\podd{\mathbb{P}^{(k)}_{2t+1}}
\def\peven{\mathbb{P}^{(k)}_{2t+2}}
\def\TT{{\mathbb T}}
\def\bbT{{\mathbb T}}
\def\cE{{\mathcal E}}

\def\wt{\widetilde}
\def\wh{\widehat}
\voffset=-0.5in

\setstretch{1.1}
\pagestyle{myheadings}
\markright{{\small \sc F\"uredi and Jiang:}
  {\it\small Tur\'an numbers of $k$-uniform linear cycles}}

\title{\huge\bf Hypergraph Tur\'an numbers of linear cycles}

\author{
Zolt\'an F\"uredi\thanks{R\'enyi Institute of Mathematics, Budapest, Hungary.
E-mail: z-furedi@illinois.edu
\newline
Research supported in part by the Hungarian National Science Foundation OTKA,
 and by a European Research Council Advanced Investigators Grant 267195.
}\quad\quad
Tao Jiang\thanks{Dept. of Mathematics, Miami University, Oxford,
OH 45056, USA. E-mail: jiangt@miamioh.edu.
  %
} }

\date{Feb 10, 2013}

\maketitle
\begin{abstract}
A {\it $k$-uniform linear cycle} of length $\ell$, denoted by
$\ckl$, is a cyclic list of $k$-sets $A_1,\ldots, A_{\ell}$ such that
consecutive sets intersect in exactly one element and nonconsecutive
sets are disjoint.
For all $k\geq 5$ and $\ell\geq 3$ and sufficiently large $n$ we
detemine the largest size of a $k$-uniform set family on $[n]$ not containing a linear
cycle of length $\ell$. For odd $\ell=2t+1$ the unique extremal family $\cF_S$ consists of
all $k$-sets in $[n]$ intersecting a fixed $t$-set $S$ in $[n]$. For even $\ell=2t+2$,
the unique extremal family consists of $\cF_S$ plus all the $k$-sets outside $S$
containing some fixed two elements.  For $k\geq 4$ and large $n$ we also establish
an exact result for so-called {\it minimal cycles}.
For all $k\geq 4$ our results substantially  extend Erd\H{o}s' result on largest
 $k$-uniform families without $t+1$ pairwise disjoint members and confirm,
 in a stronger form, a conjecture of  Mubayi and Verstra\"ete~\cite{MV1}.
Our main method is the delta system method.
\end{abstract}

\section{Introduction}\label{s:1}

The delta system method is a very useful tool for set system problems.
It was fully
  developed in a series of papers including ~\cite{furedi-1983} and ~\cite{frankl-furedi-exact}.
It was successfully used for starlike configurations in ~\cite{frankl-furedi-exact} and~\cite{lale-furedi}
and recently also for larger configurations (as paths and trees) in~\cite{furedi-2012} and ~\cite{FJS}.
In this paper we apply the delta system method,
 particularly tools from ~\cite{furedi-1983} and~\cite{frankl-furedi-exact},  to determine,
  for all $k\geq 5$ and large $n$, the Tur\'an numbers of certain hypergraphs called
 {\it $k$-uniform linear cycles}.
This confirms, in a stronger form, a conjecture of Mubayi and Verstra\"ete \cite{MV1} for $k\geq 5$ and
 adds to the limited list of hypergraphs whose Tur\'an numbers
 have been known either exactly or asymptotically.

We organize the paper as follows. Section~\ref{s:2} and ~\ref{s:3} contain definitions
concerning hypergraphs.
Section~\ref{s:4} gives a rough upper bound establishing the correct order of the magnitude.
Section~\ref{s:6} contains the statements of the main results.
Section~\ref{s:7} introduces the delta system method and lemmas needed for the linear cycle
 problem and Sections~\ref{s:8}--\ref{s:10} contain proofs.

\section{Definitions: shadows, degrees, delta systems}\label{s:2}
A {\it hypergraph} $\cF=(V,\cE)$ consists of a set $V$ of vertices and a set $\cE$ of edges,
where each edge is a subset of $V$.
If $V$ has $n$ vertices, then it is often convenient to just assume that
 $V=[n]=\{1,2,\ldots, n\}$.
Let $\binom{V}{k}$ denote the collection of all the $k$-subsets of $V$.
If all the edges of $\cF$ are $k$-subsets of $V$,
 then we write $\cF\subseteq \binom{V}{k}$ and say that $\cF$ is a {\it $k$-uniform} hypergraph,
 or a {\it $k$-graph} for brevity, on $V$.
Note that the usual graphs are precisely $2$-graphs on respective vertex sets.
A hypergraph $\cF=(V,\cE)$ is also often times called a {\it set system} or {\it set family}
on $V$ with its edges referred to as the {\it members} of the set system/family.
A $k$-graph $\cF$ is {\it $k$-partite} if its vertex set $V$ can be partitioned into $k$
subsets $V_1,\ldots, V_k$ such that each edge of $\cF$ contains precisely one vertex from
each $V_i$.

The {\it shadow} of $\cF$, denoted by $\partial(\cF)$, is defined as
$$\partial(\cF)=\{D:  \exists F\in \cF, D\subsetneq F\}.$$
Here, we treat $\emptyset$ as a member of $\partial(\cF)$.
We define the {\it $p$-shadow} of $\cF$ to be
$$\partial_p(\cF)=\{D: D\in \partial(\cF), |D|=p\}.$$
The Lov\'asz'~\cite{L79} version of the Kruskal-Katona theorem states that if
 $\cF$ is a $k$-graph of size $|\cF|=\binom{x}{k}$ where $x\geq k-1$ is a real number,
 then for $k\geq p\geq 1$
\begin{equation}\label{eq:KK1.1}
  |\partial_p(\cF)|\geq \binom{x}{p}.
  \end{equation}

Let $\cF$ be a hypergraph on $[n]$ and $D\subseteq V(\cF)$.
The {\it degree } $\deg_\cF(D)$ of $D$ in $\cF$,  is defined as
$$\deg_\cF(D)=|\{F: F\in \cF, D\subseteq F\}|.$$

A family of sets $F_1,\ldots, F_s$ is said to form an {\it $s$-star} or {\it $\Delta$-system}
of size $s$ with {\it kernel} $D$ if $F_i\cap F_j=D$ for all $1\leq i<j\leq s$ and
$\forall i\in [s], F_i\setminus D\neq \emptyset$.
The sets $F_1,\ldots, F_s$ are called the {\it petals} of this $s$-star.
Note that we allow $D=\emptyset$.
Let $\cF$ be a hypergraph and $D\subseteq V(\cF)$.
The kernel degree $\deg^*_\cF(D)$ of $D$ in $\cF$ is defined  as
$$\deg^*_\cF(D)=\max \{s: \cF \mbox{ contains an $s$-star with kernel } D\}.$$

Given a $k$-graph $\cF$ on a set $V$ and a positive integer $s$, the {\it kernel graph} of $\cF$
 with threshold $s$, denoted by $\Ker_s(\cF)$, is defined as
$$\Ker_s(\cF)=\{D\subseteq V: \deg^*_\cF(D)\geq s\}.$$
For convenience, if $D\in \Ker_s(\cF)$, we will just say that $D$ {\it is a kernel}.
For each $1\leq r\leq k-1$, the $r$-kernel graph of $\cF$ with threshold $s$,
denoted by $\Ker^{(r)}_s(\cF)$, is defined as
$$\Ker^{(r)}_s(\cF)=\{D\subseteq V: |D|=r, \deg^*_\cF(D)\geq s\}.$$
If $D\in \Ker^{(r)}_s(\cF)$ we will just say that $D$ is an {\it  $r$-kernel}.
Throughout the paper, we will frequently use the following fact which follows easily
from the definition of $\deg^*_\cF(D)$.
\begin{equation}\label{eq:1.2}
 \text{Given sets $D,Y$, if $\deg^*_\cF(D) > |Y|$ then } \exists F\in \cF
  \text{ such that }D\subseteq F \text{ and }(F\setminus D)\cap Y=\emptyset.
  \end{equation}


\section{Matchings, intersecting hypergraphs, paths and cycles} \label{s:3}

Given $n,k,t$, let $\cF_S$  be the $k$-graph on $[n]$ formed by
 taking a $t$-set $S$ in $[n]$ and taking as edges all the $k$-sets in $[n]$ that intersect $S$.
Clearly, $\cF_S$  contains no $t+1$ pairwise disjoint members, i.e., its matching number is $t$.
Erd\H{o}s~\cite{erdos-matching} showed that
there is a smallest positive integer $n_0(k,t)$ such that, for all $n> n_0(k,t)$, $\cF_S$
 is the largest  $k$-uniform set system on $[n]$ not containing $t+1$
pairwise disjoint members.
The function $n_0(k,t)$ has not been completely determined.
The value of $n_0(k,2)$ was determined in the classical Erd\H{o}s-Ko-Rado Theorem~\cite{EKR}
 about intersecting families.
For $k=2$ (graphs) the value of $n_0(2,t)$  was determined by Erd\H{o}s and Gallai~\cite{EG65}.
The case $k=3$ was recently investigated by Frankl, R\"odl, and Rucin\'ski~\cite{FRR} and
 $n_0(3,t)$ was finally determined by {\L}uczak and Mieczkowska~\cite{LM} for large $t$,
 and by Frankl~\cite{F12} for all $t$.
In general, Huang, Loh, and Sudakov~\cite{HLS} showed $n_0(k,t)< 3tk^2$,
 which was slightly improved in~\cite{FLM} and greatly improved to $n_0(k,t)\leq (2t+1)k-t$
 by Frankl~\cite{frankl-matching}.

Frankl~\cite{frankl2} showed that for every $n,k,t$ if a $k$-graph $\cF$ on $[n]$ has no
  $t+1$ pairwise  disjoint edges then $|\cF|\leq t\binom{n-1}{k-1}$.
This implies
\begin{equation}\label{eq:2.3}
  \forall D\subseteq [n]  \text{ if } \deg^*_\cF(D) \leq s, \text{ then }\deg_\cF(D)\leq s\binom{n-|D|-1}{k-|D|-1}.
  \end{equation}

Frankl~\cite{frankl1} considered set systems that do not contain
two members intersecting in exactly one element.
This condition is equivalent to
forbidding a linear path of length $2$ (see definition below).
He showed that for all $k\geq 4$ there exists a bound $m(k)$ such that
\begin{equation}\label{eq:EKR}
\text{ if } \cH\subseteq\binom{[m]}{k}  \text { satisfies }\forall A,B\in \cH,
|A\cap B|\neq 1 \text{ and }m> m(k), \text{ then } |\cH|\leq \binom{m-2}{k-2}.
  \end{equation}
The unique extremal family is obtained by taking as members
all the $k$-sets in $[m]$ containing a fixed set of two elements.

We now introduce some notions of hypergraph paths and cycles.
While the notion of a hypergraph matching is a straightforward extension of that
 of a graph matching, there are different possibilities for paths and cycles.
We discuss three versions, Berge path, minimal path and linear (or loose) path.
A {\it Berge path} of length $\ell$ in the hypergraph $\cF$ is a list of distinct hyperedges
  $F_1,\ldots, F_\ell\in \cF$ and $\ell+1$ distinct vertices $P:=\{ v_1,\ldots, v_{\ell+1}\}$
 such that for each $1\leq i\leq \ell$, $F_i$ contains $v_i$ and $v_{i+1}$.

If we allow  only consecutive $F_i$'s to intersect, i.e.,  $F_i\cap F_j=\emptyset$ when $|i-j|\geq 2$,
 then the resulting Berge path is called a {\it minimal} path.
We denote the family of all $k$-uniform minimal paths of length $\ell$ by $\cpkl$.
If we require all the $F_i$'s to be pairwise disjoint
 outside $P$ and $F_i\cap P=\{ v_i, v_{i+1} \}$, then the path is unique.
We call it the $k$-uniform {\it linear path} of length $\ell$ and denote it by $\pkl$.
Note that $\pkl$ is a member of $\cpkl$.

Likewise, a $k$-uniform {\it Berge cycle} of length $\ell$ is a cyclic list of distinct $k$-sets
$F_1,\ldots, F_\ell$ and $\ell$ distinct vertices $C=\{ v_1,\ldots, v_\ell\}$ such that
for each $1\leq i\leq \ell$, $F_i$ contains $v_i$ and $v_{i+1}$ (where $v_{\ell+1}=v_1$).
If we allow only consecutive $F_i$'s in the cyclic list to intersect
 then the resulting cycle is called a {\it minimal} cycle.
We denote the family of all $k$-uniform minimal cycles of length $\ell$ by $\cckl$.
If we require all the $F_i$'s to be pairwise disjoint outside $C$
 and $F_i\cap C=\{ v_i, v_{i+1}\}$, then
the cycle is unique and we call it the $k$-uniform {\it linear cycle} of length $\ell$ and
denote it by $\ckl$. Note that $\ckl$ is a member of $\cckl$.

The {\it triangulated cycle} $\TT_\ell^{(3)}$, is a triple system on $2\ell$ vertices
  $\{ v_1, \dots, v_\ell, u_1, \dots, u_\ell\}$ with $2\ell-2$ edges
 $A_i:=\{ v_1,v_i, v_{i+1}\}$ ($1< i <\ell$) and $B_j:=\{ v_j, v_{j+1}, u_j\}$
 ($\ell\geq 3$). Note that the $B_j$'s form $\mathbb{C}^{(3)}_\ell$.


\section{Hypergraph extensions and an estimate of the Tur\'an number}\label{s:4}

Given a hypergraph $\cH$ whose edges have size at most $k$,
the {\it $k$-expansion} of $\cH$, denoted by $\cH^{(k)}$,
 is the $k$-graph obtained by enlarging each edge of $\cH$ into a $k$-set
 by using new vertices (called {\it expansion vertices})  such that different edges
 are enlarged using disjoint sets of expansion vertices.
For instance, if $\cH=\{1, 12, 123\}$, then $\{1ab, 12c,  123\}$ is the $3$-expansion of
 $\cH$ and $\{1abc, 12de, 123f\}$ is the $4$-expansion of $\cH$.
Note  that for any $k,b$ where $b\geq 2$ and $k\geq b+1$,
 the $k$-expansion of a $b$-uniform linear (or minimal) $\ell$-cycle is a
 $k$-uniform linear (or minimal) $\ell$-cycle.

\begin{proposition}   \label{expansion}
Let $k$ be a positive integer.
Let $\cH:=\{ E_1, \dots, E_t\}$ be a hypergraph whose edges are sets of size at most $k$.
Let $\cF$ be a $k$-graph, $s=tk$.
If $\cH\subseteq \Ker_{s}(\cF)$, then $\cH^{(k)}\subseteq \cF$.
\end{proposition}
\begin{proof}
We want to expand the edges of $\cH$  into edges $F_1, \dots, F_t$ of $\cF$
 such that different edges of  $\cH$ are enlarged through disjoint sets of expansion vertices.
We find $F_i$'s one by one by using \eqref{eq:1.2}.
Suppose $F_1,\ldots, F_{i-1}$ have been defined.
Let $Y=(\cup_{j<i} F_j) \cup (\cup_{j} E_j)$. Since $\deg^*_\cF(E_i)\geq s>|Y|$,
by  \eqref{eq:1.2} one can find an  $F_i\in \cF$
 such that $F_i\supseteq E_i$ and  $F_i\cap Y = E_i$. We do this  for $i=1,\ldots, t$.
The $F_i$'s form a copy of $\cH^{(k)}$.
\end{proof}

\begin{proposition}\label{forest-upper_new3}
Suppose that $\cF$ is a triple system not containing $\cthreel$.
Then $|\cF|\leq (2\ell-3)\binom{n}{2}$.

  \end{proposition}
\begin{proof} Starting with $\cF$,
whenever we can find a pair $\{x,y\}$ such that the number of triples
containing the pair is at least one and at most $2\ell-3$,
we remove all triples containing the pair from the system.
Repeat this process until no more triple can be removed. Let $\cH$ be
the remaining triple system. If $\cH\neq\emptyset$, then we must have
$\deg_\cH(\{x,y\})\geq 2\ell-2$ for all $\{x,y\}\in \partial_2(\cH)$. Clearly
$|\cF\setminus \cH|\leq (2\ell-3)\binom{n}{2}$.
We claim that $\cH=\emptyset$.
Otherwise, starting with any triple $A_1=\{v_1,v_2,v_3 \}\in \cH$ we can embed one by one
  $A_1, \dots, A_{\ell-2}$ then $B_1, \dots, B_\ell$, the edges of a triangulated
 cycle $\TT_\ell^{(3)}$ in $\cH$ in the same way as we did in the proof of Proposition~\ref{expansion}.
But $\TT_\ell^{(3)}$ contains $\mathbb{C}_\ell^{(3)}$. This contradicts $\mathbb{C}_\ell^{(3)}
\not\subseteq \cF$.
\end{proof}

Given a family $\cH$ of $k$-graphs, the {\it Tur\'an number} of $\cH$, for fixed $n$,
denoted by $\ex_k(n,\cH)$, is the maximum number edges in a $k$-graph on $[n]$ that does
not contain any member of $\cH$ as a subgraph.
If $\cH$ consists of a single $k$-graph $H$, we will write $\ex_k(n,H)$ for $\ex_k(n,\{H\})$.

\begin{corollary}\label{cycle-bound}
For all $n$ and $k, \ell\geq 3$ we have
$$\ex_k(n,\ckl)\leq (k\ell -1)\binom{n}{k-1}.$$
  \end{corollary}
\begin{proof}
Consider an $\cF\subseteq \binom{[n]}{k}$ avoiding $\ckl$.
Let $s=k\ell$.
By Proposition~\ref{expansion} the triple-system $\Ker_s^{(3)}(\cF)$ does not contain
 $\cthreel$ so we can apply Proposition~\ref{forest-upper_new3} for its size.
Use the upper bound \eqref{eq:2.3} for the degrees of the other triples of $[n]$.
We obtain
\begin{equation*}
  |\cF|\binom{k}{3}=\sum_{|T|=3,\, T\subseteq [n]} \deg_\cF(T)
      \leq |\Ker_s^{(3)}(\cF)|\binom{n-3}{k-3} + \binom{n}{3}(s-1)\binom{n-4}{k-4}.
       \end{equation*}
An easy calculation completes the proof.  \end{proof}


\section{Some previous results and a conjecture}\label{s:5}

For the class of $k$-uniform Berge paths of length $\ell$,
 Gy\H{o}ri et al.~\cite{gyori} determined $\ex_k(n, \bkl)$ exactly for
 infinitely many $n$.
For the Tur\'an problem for  $k$-uniform minimal paths of length $\ell$,
 observe that to forbid such a path it suffices to forbid a matching of size $t+1$, $\cmkt$,
  where $t=\fl{(\ell-1)/2}$.
So $\ex_k(n,\cpkl)\geq \ex_k(n, \cmkt)\geq \binom{n}{k}-\binom{n-t}{k}$, where the last lower
 bound is attained by taking all the $k$-sets in $[n]$ intersecting some fixed $t$-set $S$.
Mubayi and Verstraete~\cite{MV1} showed that this lower bound is tight up to a factor of $2$. Note
that $\binom{n}{k}-\binom{n-t}{k}=t{n-1\choose k-1}+O(n^{k-2})$.
They proved that if $k,\ell\geq 3$, $t=\fl{(\ell-1)/2}$ and $n\geq (\ell+1)k/2$, then
 $\ex_k(n,\cP^{(k)}_3)={n-1\choose k-1}$ and for $\ell,k>3$
\begin{equation} \label{mv}
t{n-1\choose k-1}+O(n^{k-2})\leq \ex_k(n,\cP^{(k)}_\ell)\leq 2t{n-1\choose k-1}+O(n^{k-2}).
  \end{equation}
Using the delta system method, F\"uredi, Jiang, and Seiver~\cite{FJS} were able to sharpen \eqref{mv}
 to determine the exact value of $\ex_k(n,\cpkl)$ for all $k\geq 3, t\geq 1$ and sufficiently large $n$
\begin{equation}\label{Jiang-Seiver}
\ex_k(n,\cP^{(k)}_{2t+1})=\binom{n}{k}-\binom{n-t}{k}, \quad \mbox{and } \quad
\ex(n,\cP^{(k)}_{2t+2})=\binom{n}{k}-\binom{n-t}{k}+1.
  \end{equation}
For $\cP^{(k)}_{2t+1}$, the only extremal family consists of all the $k$-sets in $[n]$
 that intersect some fixed $t$-set $S$. For $\cP^{(k)}_{2t+2}$,
 the only extremal family consists of all the $k$-sets in $[n]$ that intersect
 some fixed set $S$ of $t$ vertices plus one additional $k$-set that is disjoint from $S$.

The Tur\'an problem for a linear path $\pkl$ was also solved in~\cite{FJS} for all $k\geq 4$
 and sufficiently large $n$
\begin{equation} \label{FJS-path}
 \ex_k(n,\podd)=\binom{n}{k}-\binom{n-t}{k}, \quad \mbox{and} \quad
  \ex_k(n,\peven)=\binom{n}{k}-\binom{n-t}{k}+\binom{n-t-2}{k-2}.
  \end{equation}
For $\podd$, the only extremal family consists of all
the $k$-sets in $[n]$ that meet some fixed $k$-set $S$.  For $\peven$,
the only extremal family consists of all
the $k$-sets in $[n]$ that intersect some fixed $t$-set $S$ plus all the $k$-sets in $[n]\setminus S$
that contain some two fixed elements.

For minimal cycles of length $\ell$, the same lower bound of $\binom{n}{k}-\binom{n-t}{k}$
 for $\cmkt$ applies,
where $t=\fl{(\ell-1)/2}$.
Answering a conjecture of Erd\H{o}s, Mubayi and Verstra\"ete~\cite{MV2} showed that
for all $k\geq 3$ and $n\geq 3k/2$, we have $\ex_k(n,\cC^{(k)}_3)=\binom{n-1}{k-1}$.
Later for general minimal cycles they~\cite{MV1} showed that the lower bound
for $\cmkt$ is tight up to a factor of $3$.
For $k\geq 3$ $,\ell\geq 4$, $t=\fl{(\ell-1)/2}$ they have
$\ex_3(n,\cC^{(3)}_\ell)\leq \frac{5\ell-1}{6}\binom{n}{2}$,
$\ex_4(n, \cC^{(4)}_\ell)\leq \frac{5\ell}{4}\binom{n}{3}$ and
$ \ex_k(n,\cC^{(k)}_4)=\binom{n-1}{k-1}+O(n^{k-2}) $.
For $k,\ell\geq 5$, they obtained
\begin{equation} \label{mv2}
t{n-1\choose k-1}+O(n^{k-2})\leq \ex_k(n,\cckl)\leq 3t{n-1\choose k-1}+O(n^{k-2}).
 \end{equation}
For $k,\ell\geq 3$, Mubayi and Verstra\"ete~\cite{MV1} conjectured their lower bound to be asymptotically tight.
\begin{conjecture} {\rm \cite{MV1}} \label{mv-conjecture}
Let $n,k,\ell\geq 3$ be integers and $t=\fl{\frac{\ell-1}{2}}$. Then as $n\to\infty$
$$ex_k(n,\cckl)=t\binom{n-1}{k-1}+O(n^{k-2}).$$
\end{conjecture}


\section{Main results: Tur\'an numbers of  cycles}\label{s:6}

As our main result, in Theorem~\ref{main}
we determine for all $k\geq 5$ and sufficiently large $n$
the exact value of the Tur\'an number of the linear cycle $\ckl$.
In Theorem~\ref{k=4}, we determine the exact Tur\'an numbers
of  minimal cycles $\cckl$ for all $k\geq 4$ and large $n$.
Theorem \ref{k=4}
confirms the truth of Conjecture~\ref{mv-conjecture} for all $k\geq 4$ in a stronger sense.
For $k\geq 5$ and odd $\ell$, Theorem~\ref{main} is even stronger than Theorem~\ref{k=4}.

\begin{theorem} {\bf(Main result)} \label{main}
Let $k,t$ be positive integers, $k\geq 5$. For sufficiently large $n$, we have
$$\ex_k(n,\codd)=\binom{n}{k}-\binom{n-t}{k}, \quad
  \mbox{and} \quad \ex_k(n,\ceven)=\binom{n}{k}-\binom{n-t}{k}+\binom{n-t-2}{k-2}.$$
For $\codd$, the only extremal family consists of all
the $k$-sets in $[n]$ that meet some fixed $k$-set $S$.  For $\ceven$,
the only extremal family consists of all
the $k$-sets in $[n]$ that intersect some fixed $t$-set $S$ plus all the $k$-sets in $[n]\setminus S$
that contain some two fixed elements.
\end{theorem}

Note that the case $\ell=3$ was already proved in~\cite{frankl-furedi-exact}.

\begin{theorem} \label{k=4}
Let $t$ be a positive integer, $k\geq 4$.
For sufficiently large $n$, we have
$$\ex_k(n,{\cal C}^{(k)}_{2t+1})=\binom{n}{k}-\binom{n-t}{k}, \quad \mbox{and} \quad
\ex_k(n,{\cal C}^{(k)}_{2t+2})=\binom{n}{k}-\binom{n-t}{k}+1.
 $$
For ${\cal C}^{(k)}_{2t+1}$, the only extremal family consists of all
the $k$-sets in $[n]$ that meet some fixed $k$-set $S$.  For ${\cal C}^{(k)}_{2t+2}$,
the only extremal family consists of all
the $k$-sets in $[n]$ that intersect some fixed $t$-set $S$ plus one additional $k$-set outside $S$.
\end{theorem}

Our method does not work for $k=3$, however,
 we were informed that Kostochka, Mubayi, and Verstra\"ete \cite{KMV}
 have some new results on this case.

Note that the answers in the main theorem are exactly the same as for $\podd$ and $\peven$
 with the same extremal constructions as well.
However, neither the path result nor the cycle result imply each other and the proofs for cycles
are more involved  and require additional ideas.


\section{The delta-system method}\label{s:7}

In this section, we introduce our main tools we need from the delta-system method.
Given a hypergraph
$\cF$ and an edge $F$ of $\cF$, we define the {\it intersection structure} of $F$ relative to
$\cF$ to be
$$\cI(F,\cF)=\{F\cap F': F'\in \cF, F'\neq F\}.$$

Let $\cF$ be a $k$-partite $k$-graph with a $k$-partition
$(X_1,\ldots, X_k)$. Hence, each edge of $\cF$ contains
 exactly one element of each $X_i$.
Given any subset $S$ of $[n]$, let
$$\Pi(S)=\{i: S\cap X_i\neq \emptyset\} \subseteq [k].$$
So $\Pi(S)$ records which parts in the given $k$-partition that $S$ meets.
If $\cL$ is a collection of subsets of $[n]$, then we define
$$\Pi(\cL)=\{\Pi(S): S\in \cL\} \subseteq 2^{[k]}.$$
We will call $\Pi(\cI(F,\cF))$ the {\it intersection pattern} of $F$ relative to $\cF$.
Given $F\in \cF$ and $I\subseteq [k]$, let $F[I]=F\cap(\bigcup_{i\in I} X_i)$.
So $F[I]$ is the restriction of $F$ onto those parts indexed by $I$.

\begin{lemma}{\bf (The intersection semilattice lemma~\cite{furedi-1983})}  \label{homogeneous}
For any positive integers $s$ and $k$, there exists a positive constant $c(k,s)$ such that
{\bf every}
family $\cF\subseteq {[n]\choose k}$ contains a subfamily $\cF^*\subseteq \cF$ satisfying
\begin{enumerate}
\item $|\cF^*|\geq c(k,s)|\cF|$.
\item $\cF^*$ is $k$-partite, together with a $k$-partition $(X_1,\ldots, X_k)$.
\item There exists a family $\cJ$ of proper subsets of $[k]$ such that
  $\Pi(\cI(F,\cF^*))=\cJ$ holds for all $F\in \cF^*$.
\item $\cJ$ is closed under intersection, i.e.,
 for all $I,I'\in \cJ$ we have $I\cap I'\in \cJ$ as well.
\item For every $F\in \cF^*$, and every $I\in \cJ$, $F[I]\in \Ker_s(\cF)$.
\end{enumerate}
\end{lemma}

\begin{definition}
{\rm We call a family $\cF^*$ that satisfies items (2)-(5) of Lemma~\ref{homogeneous}
 {\it $(k,s)$-homogeneous} with intersection pattern $\cJ$.}
\end{definition}

Given a family $\cL$ of subsets of $[k]$, the {\it rank} of $\cL$ is the minimum size
 of a set in $[k]$ that is not contained in any member of $\cL$. Formally
$$r(\cL)=\min\{|D|: D\subseteq [k], \not\exists B\in \cL, D\subseteq B\}.$$

The next three lemmas were used in many earlier papers, e.g.,
 we can refer to~\cite{frankl-furedi-exact, FJS}.

\begin{lemma}{\bf (The rank bound)} \label {rank-bound}\enskip
Let $k,s$ be positive integers. Let $\cF^*$ be a
$(k,s)$-homogeneous family on $n$ vertices with intersection
pattern $\cJ$. If $r(\cJ)=p$, then $|\cF^*|\leq {n\choose p}$.
\end{lemma}

\begin{lemma} \label{pattern-structure}
Let $k\geq 3$ be a positive integer. Let $\cL$ be a family of proper subsets of $[k]$ that is
closed under intersection.
\begin{enumerate}
\item If $\cL$ has rank $k$, then it contains all the proper subsets of $[k]$.
\item If $\cL$ has rank $k-1$, then the elements of $[k]$ can be listed as
 $x_1,x_2,\ldots, x_t, x_{t+1},\ldots, x_k$ such that for every
 $t+1\leq i\leq k$, $[k]\setminus \{x_i\} \in \cL$ and for all $1\leq i<j\leq t$, $[k]\setminus
\{x_i,x_j\}\in \cL$. If $t=1$, then we say that $\cL$ is of {\bf type 1}. If
$t\geq 2$,  then we say that $\cL$ is of {\bf type 2}.
If $\cL$ is of type 1, then there exists an element $x\in [k]$ such that $\cL$ contains
 all the proper subsets of $[k]$ that contains $x$; we call $x$ the {\bf central element} of $\cL$.
If $\cL$ is of type 2, then
 $\forall C\subseteq \{x_1,\ldots, x_t\}, \forall D\subseteq \{x_{t+1},\ldots, x_k\}$,
 where $|C|\leq t-2$, we have $C\cup D\in \cL$.
\end{enumerate}
\end{lemma}

\begin{lemma}\label{partition1} {\bf (The partition lemma)}\enskip
Let $n,k,s$ be positive integers, let $\cF\subseteq {[n]\choose k}$.
Then $\cF$ can be partitioned into subfamilies $\cG_1,\cG_2,\dots, \cG_{m}$ and $\cF_0$ such that
 $|\cF_{0}|\leq \frac{1}{c(k,s)}{n\choose k-2}$ and for $1\leq i\leq m$
 each $\cG_i$ is $(k,s)$-homogeneous with intersection pattern $\cJ_i$
 of rank at least $k-1$.
\end{lemma}

\begin{proof}
Apply Lemma~\ref{homogeneous} to $\cF$ to get a $(k,s)$-homogeneous subfamily $\cG_1$ with
intersection pattern $\cJ_1$ such that $|\cG_1|\geq c(k,s) |\cF|$.
Then apply Lemma~\ref{homogeneous} again to $\cF\setminus \cG_1$
 to get a $(k,s)$-homogeneous subfamily $\cG_2$ with intersection pattern $\cJ_2$ such
 that $|\cG_2|\geq c(k,s)|\cF\setminus \cG_1|$. We continue like this.
Let $m$ be the smallest nonnegative integer such that $\cJ_{m+1}$ has rank $k-2$ or less
 and let $\cF_0= \cF \setminus (\cup_{i\leq m} \cG_i)$.
By our procedure, $|\cG_{m+1}|\geq c(k,s)|\cF_0|$ and Lemma~\ref{rank-bound} gives the upper bound.
\end{proof}


\section{Homogeneous families without cycles are not of type 2}\label{s:8}

The aim of this section is to describe the typical intersection structures
 of the members of a $k$-uniform hypergraph avoiding cycles.

\begin{definition}
{\rm A set-family $\cF$ is {\it centralized} with threshold $s$ if $\forall F\in \cF$ there exists
an element $c(F)\in F$ such that if $D$ is a proper subset of $F$ containing $c(F)$ then $D\in \Ker_s(\cF)$.
We call $c(F)$ a {\it central element} of $F$. (The choice of $c(F)$ may not be unique, but we will fix one.)}
\end{definition}

\begin{theorem}\label{partition2} {\bf (The partition theorem)}
Let $k,\ell, s$ be positive integers, where $k\geq 4$, $\ell\geq 3$, and $s\geq k\ell$.
Let $\cF\subseteq {[n]\choose k}$. If $k=4$, then suppose $\cF$ contains no member
of ${\cal C}^{(4)}_\ell$.
If $k\geq 5$, then suppose $\ckl\not\subseteq\cF$.
Then $\cF$ can be partitioned into subfamilies $\cF_1,\cF_0$ such that $\cF_1$ is
centralized with threshold $s$ and  $|\cF_0|\leq \frac{1}{c(k,s)}{n\choose k-2}$.
\end{theorem}

The proof consists of several small steps and is given at the end of this Section.

The following proposition follows immediately from Lemma \ref{homogeneous} and Lemma \ref{pattern-structure}.

\begin{proposition} \label{centralized}
If $\cF=\bigcup_{i=1}^m \cG_i$, where $\forall i\in [m]$, $\cG_i$  is a $(k,s)$-homogeneous family
 whose intersection pattern $\cJ_i$ has rank $k-1$ and is of type 1,
 then $\cF$ is centralized with threshold $s$.\qed
\end{proposition}

Recall that given a set $S$, $2^S$ denotes the collection of all subsets of $S$.

\begin{lemma} \label{kernel-clique}
Let $k\geq 5$ be an integer and let $\cL$ be a family of subsets of $[k]$ that is closed under
 intersection and has rank $k-1$ and is of type $2$.
 Then there exists $S\subseteq [n]$ such that $|S|=3$ and $2^S\subseteq \cL$.
\end{lemma}
\begin{proof}
Let $S$ be a  $3$-subset of $[k]\setminus\{x_1,x_2\}$.
Any subset $A$ of $S$ can be written as $C\cup D$, where $C\subseteq \{x_1,\ldots, x_t\}$, $|C|\leq t-2$,
 and $D\subseteq \{x_{t+1},\ldots, x_k\}$.
By Lemma \ref{pattern-structure}, $A\in \cL$.
\end{proof}

For $k=4$, we prove something a bit weaker.

\begin{proposition} \label{4-kernel-cycle}
Let $\cL$ be a family of subsets of $[4]$ that is closed under intersection and has rank $3$
and is of type $2$. Then $\cL$ contains a minimal $3$-cycle where each edge has size $2$ or $3$.
\end{proposition}
\begin{proof}
By Lemma \ref{pattern-structure},
 $\forall C\subseteq \{x_1,\ldots, x_t\}$, where $|C|\leq t-2$, and  $\forall D\subseteq
\{x_{t+1},\ldots x_k\}$, we have $C\cup D\in \cL$.
If $t=4$, then we have $\{x_1,x_2\}$, $\{x_1,x_3\}$, $\{x_2,x_3\}\in \cL$.
If $t=3$, then we have $\{x_1,x_2,x_3\}$, $\{x_1,x_4\}$, $\{x_3, x_4\}\in \cL$.
If $t=2$, then we have $\{x_1,x_2, x_3\}$, $\{x_1, x_2, x_4\}$, $\{x_3, x_4\}\in \cL$.
\end{proof}

\begin{lemma} \label{cycle-corollary}
Let $k,\ell,s$ be positive integers, where $k\geq 5, \ell\geq 3$, and $s\geq k\ell$.
Let $\cG$ be a $(k,s)$-homogeneous family with a $k$-partition
$(X_1,\ldots, X_k)$ and intersection pattern $\cJ$ such that either $\cJ$ has
rank $k$ or has rank $k-1$ and is of type 2. Then $\ckl\subseteq \cG$.
\end{lemma}
\begin{proof}
By Lemma \ref{pattern-structure} and Lemma \ref{kernel-clique},
 there exists a $3$-set $S\subseteq [k]$ such that $2^S\subseteq \cJ$.
By definition, this means that $\forall F\in \cG , \forall A\subseteq S$,
 $F\cap (\cup_{i\in A} X_i)$ is  a member of $\Ker_s(\cG)$.
Let $\cH=\{F\cap (\cup_{i\in A} X_i): F\in \cG, A\subseteq S\}$.
Then $\cH\subseteq \Ker_s(\cG)$. Note that $\cH$ is down-closed.
Let $D\in \partial_2(\cH)$. Then $D\in \Ker_s(\cF)$.
So $\cF$ contains an $s$-star $\cL$ with kernel $D$.
The restriction of the $s$ petals of $\cL$ on $\bigcup_{i\in S} X_i$ are $s$ distinct
 triples in $\cH$ containing $D$. So $\deg^*_\cH(D)=\deg_\cH(D)\geq s$ for all $D\in \partial_2(\cH)$.
 This allows us to embed the triangulated cycle $\TT_\ell^{(3)}$ into $\cH$
  as we did in the proof of Proposition~\ref{forest-upper_new3}.
Since $\cthreel\subseteq \TT_\ell^{(3)}$ and $\cH\subseteq \Ker_s(\cG)$
 Proposition \ref{expansion} implies that
$\cG$ contains a $k$-expansion of $\cthreel$, which is $\ckl$.
\end{proof}

For $k=4$, Lemma \ref{4-kernel-cycle} and induction yield

\begin{proposition} \label{4cycle-corollary}
Let $\ell,s$ be positive integers, where $\ell \geq 3$ and $s\geq 4\ell$.
Let $\cG$ be a $(4,s)$-homogeneous family with a $4$-partition $(X_1,\ldots, X_4)$
 intersection pattern $\cJ$ such that either $\cJ$ has
rank $4$ or has rank $3$ and is of type 2.
Then $\cG$ contains a member of ${\cal C}^{(4)}_\ell$.
\end{proposition}

\begin{proof}
By Lemma \ref{4-kernel-cycle} $\cJ$ contains a minimal $3$-cycle $L$.
Consider first the case where the edges of $L$ are $I_1=\{1,2,3\}, I_2=\{1,2,4\}$ and $I_3=\{3,4\}$.
Then $\forall F\in \cG, F[I_1], F[I_2], F[I_3]\in \Ker_s(\cG)$.
We use induction on $\ell$ to show that $\cG$ contains a member of ${\cal C}^{(4)}_\ell$ such
that for any two consecutive edges $E,E'$ on the cycle, either they intersect in exactly
one vertex and that vertex lies in $X_3$ or $X_4$ or they intersect in two vertices and
those two vertices lie in $X_1$ and $X_2$, respectively; we call such a member of ${\cal C}^{(4)}_\ell$
 a {\it good} member.

For the basis step let $\ell=3$.
Let $E_0=\{a_1,a_2,a_3,a_4\}$ be any edge in $\cG$,
 where $\forall i\in [4], a_i\in X_i$. By our assumption, $\{a_1,a_2,a_3\}, \{a_1,a_2,a_4\}, \{a_3,a_4\}$
 all have kernel degree at least $s\geq 4\ell$.
So we can find $E_1=\{a_1,a_2,a_3, a'_4\}, E_2=\{a_1,a_2,a'_3,a_4\}, E_3=\{a'_1,a'_2, a_3,a_4\}\in \cG$,
 where $\forall i\in [4], a'_i\in X_i$ and $a_1,\ldots, a_4, a'_1,\ldots, a'_4$ are all distinct.
 Now, $E_1,E_2,E_3$ form
a minimal $3$-cycle that satisfies the claim. For the induction step, suppose $\ell\geq 4$ and
that the claim holds for $\ell-1$ and that $\cG$ contains a good member $L$ of ${\cal C}^{(4)}_{\ell-1}$.
Let $E$ be an edge of $L$. Let $E',E''$ be the edge preceding $E$
and succeeding $E$, respectively on $L$. Then $\{|E\cap E'|, |E\cap E''|\}= \{1,1\}$ or $\{1,2\}$.
In the former case, we may assume $E\cap E'=\{b_3\}\subseteq X_3$ and $E\cap E''=\{b_4\}\subseteq X_4$.
By our assumption $E\setminus \{b_4\}=E[I_1]\in \Ker_s(\cG)$
and $E\setminus \{b_3\}=E[I_2]\in \Ker_s(\cG)$. Since $s\geq 4\ell$ and $n(L)\leq 4\ell-1$,
we can find $b'_3\in X_3\setminus V(L), b'_4\in X_4\setminus V(L)$ such that
$(E\setminus \{b_3\})\cup\{b'_3\}\in\cG$ and  $(E\setminus \{b_4\})\cup\{b'_4\}\in\cF$.
Replacing $E$ with these two members of $\cG$ in $L$ yields a good member of
${\cal C}^{(4)}_\ell$. The case where $\{|E\cap E'|, |E\cap E''|\}=\{1,2\}$ can be handled similarly.
This completes the induction.

Similar arguments apply if $\cJ$ contains other kinds of minimal $3$-cycles.
\end{proof}

\noindent{\it Proof of Theorem~\ref{partition2}.}\enskip
Consider the partition $\cF= \cG_1\cup \dots\cup \cG_{m}\cup \cF_0$
 given by the Partition Lemma~\ref{partition1}.
For each $i\in [m]$,  $\cJ_i$  has rank at least $k-1$.
If some $\cJ_i$ has rank $k$ or has rank $k-1$ and is of type $2$, then
 $\ckl\subseteq \cG_i\subseteq \cF$  by Lemma~\ref{cycle-corollary}, a contradiction
  in the case of $k\geq 5$.
So each $\cJ_i$ has rank $k-1$ and is of type $1$.
By Proposition \ref{centralized}, $\cF_1:= \bigcup_{i=1}^m \cG_i$
is centralized with threshold $s$.

For $k=4$ we use Proposition \ref{4cycle-corollary}  in place of Lemma~\ref{cycle-corollary}.
\qed

\section{The kernel structure of centralized families}\label{s:9}

\begin{theorem} \label{L-set}
Let $k,\ell$ be integers, where $k\geq 4$ and $\ell\geq 3$.
Let $t=\fl{\frac{\ell-1}{2}}$. Let $s=k\ell$.
For all $n\geq n_1(k,\ell)$
the following holds:  If $\cF\subseteq \binom{[n]}{k}$
is a centralized family with threshold $s$, $\ckl\not\subseteq \cF$, and
$|\cF|\geq t\binom{n}{k-1}-o(n^{k-1})$,
then there exist $S,T\subseteq [n]$, where $S\cap T=\emptyset$, $|S|=t$ and
$|T|\geq n-o(n)$, such that $\Ker^{(2)}_s(\cF)$ contains all the edges between $S$
and $T$.
\end{theorem}

\begin{lemma} \label{binomial-sum}
Let $n,p,q$ be positive integers and $x_1,\ldots, x_n$  reals
such that $q\geq p$ and  $x_1\geq \ldots \geq x_n\geq q-1$.
Let $M=\sum_{i=1}^n \binom{x}{p}$. Then $\forall h\in [n]$, we have
$$\sum_{i=h}^n \binom{x_i}{q} \leq p^{q-p} \frac{M^{\frac{q}{p}}}{h^{\frac{q}{p}-1}}.$$
\end{lemma}
\begin{proof}
Since $\sum_{i=1}^n \binom{x_i}{p}=M$ and $\binom{x_1}{p}\geq \binom{x_2}{p}\geq \ldots \geq \binom{x_n}{p}$,
 we have
$\binom{x_h}{p}\leq \frac{M}{h}$. Since $\binom{x_h}{p}\geq(\frac{x_h}{p})^p$, this yields
$x_h< p (\frac{M}{h})^\frac{1}{p}$.
Also, trivially for $i\geq h$, $\binom{x_i}{q}\leq x_i^{q-p}\binom{x_i}{p} \leq x_h^{q-p}\binom{x_i}{p}$.
Hence,
$$\sum_{i=h}^n\binom{x_i}{q}\leq (x_h)^{q-p}\sum_{i=h}^n \binom{x_i}{p}
   =x_h^{q-p} M< M\left [p \left(\frac{M}{h}\right)^\frac{1}{p}\right]^{q-p}=
  p^{q-p}\frac{M^\frac{q}{p}}{h^{\frac{q}{p}-1}}.$$
\end{proof}

\noindent{\it Proof of Theorem~\ref{L-set}.}\enskip
Let us partition $\cF$ according to $c(F)$.
For each $i\in [n]$,
let $$\cA_i=\{F\in \cF: c(F)=i\}, \quad \mbox{ and } \quad \cA'_i=\cF_i - \{i\}.$$
Let $D \in \partial_2(\cA'_i)$.
Then $D\cup \{i\}$ is a proper $3$-subset of $F$ containing $i=c(F)$.
Since $\cF$ is centralized with threshold $s$,   $D\cup \{i\}\in \Ker^{(3)}_s(\cF)$.
Thus $\forall D\in \partial_2(\cA'_i)$, we have $D\cup \{i\}\in \Ker^{(3)}_s(\cF)$.
This yields
\begin{equation} \label{kb-lower}
|\Ker^{(3)}_s(\cF)|\geq \frac{1}{3}\sum_{i=1}^n |\partial_2(\cA'_i)|.
\end{equation}

Since $\ckl\not\subseteq \cF$ and $s=k\ell$, by Proposition~\ref{expansion},
$\cthreel\not\subseteq \Ker^{(3)}_s(\cF)$.
By Proposition~\ref{forest-upper_new3}, we have
\begin{equation} \label{cycle-upper}
|\Ker^{(3)}_s(\cF)|\leq \ex_3(n,\cthreel) < 2\ell\binom{n}{2}.
\end{equation}
By \eqref{kb-lower} and \eqref{cycle-upper}, we have
\begin{equation}
\sum_{i=1}^n |\partial_2(\cA'_i)|\leq  6 \ell \binom{n}{2}<3\ell n^2 .
\end{equation}
For each $i\in [n]$, let $x_i\geq 1$ be the real such that
$|\partial_2(\cA'_i)|=\binom{x_i}{2}$, where without loss of generality
we may assume that $x_1\geq \ldots\geq x_n$.
Let $M=\sum_{i=1}^n \binom{x_i}{2}$. Then $M<3\ell n^2$.
By Kruskal-Katona's theorem~\eqref{eq:KK1.1}, $\forall i\in [n]$, $|\cA'_i|\leq \binom{x_i}{k-1}$.
Now, set $\e=\frac{1}{2(t+1)}$ and $h=\ce{n^\e}$.
Applying Lemma~\ref{binomial-sum} with $p=2, q=k-1$, we have
\begin{equation} \label{tail}
\sum_{i\geq h} |\cA_i|=\sum_{i=h}^n |\cA'_i|=\sum_{i=h}^n \binom{x_i}{k-1}
\leq 2^{k-3}\frac{M^{\frac{k-1}{2}}}{h^{\frac{k-3}{2}}}=O(n^{k-1-\frac{k-3}{2}\e})
=O(n^{k-1-\frac{k-3}{4(t+1)}}).
\end{equation}

Let $L=[h]$.  Let $\cF_1=\{F\in \cF: c(F)\notin L\}$.  Then $\cF_1\subseteq
\bigcup_{i> h} \cA_i$. By \eqref{tail}, we have
\begin{equation}\label{F1-upper}
|\cF_1|=O(n^{k-1-\frac{k-3}{4(t+1)}}).
\end{equation}
By our definition, $\forall F\in \cF\setminus \cF_1$ we have $c(F)\in L$.
Let $\cF_2=\{F: F\in \cF\setminus \cF_1, |F\cap L|\geq 2\}$. Then
\begin{equation} \label{F2-upper}
|\cF_2|\leq \binom{|L|}{2}\binom{n-|L|}{k-2}\leq n^{k-2+2\e}<n^{k-\frac{3}{2}}.
\end{equation}
Let $\cF'=\cF\setminus (\cF_1\cup \cF_2)$. Then $\forall F\in \cF'$, we have
$F\cap L=\{c(F)\}$.
For each $A\subseteq L$, let
$$\cF_A=\{F\in \cF': \forall a\in A, (F\setminus L)\cup \{a\}\in  \cF'\},
\quad \mbox{and}
\quad   \cF'_A=\{F\setminus L: F\in \cF_A\}.$$
Then $\cF'=\bigcup_{A\subseteq L} \cF_A$.

A set $W$ of vertices in a hypergraph $\cG$ is {\it strongly independent} if
no two vertices of $W$ lie in the same edge of $\cG$.
The cycle $\ckl$ has a strongly independent set $W$ of $t+1=1+\fl{(\ell-1)/2}
=\ce{\ell/2}$ vertices
whose removal leaves a $(k-1)$-uniform hypergraph $\cT$ with $\ell$ edges.
It is easy to see that $\cT\subseteq {\mathbb C}_{\lceil 3\ell/2\rceil}^{(k-1)}$.
By Corollary~\ref{cycle-bound} we obtain
\begin{equation}\label{upper2}
  \ex_{k-1}(n, \cT)\leq \ex_{k-1}(n,{\mathbb C}_{\lceil 3\ell/2\rceil}^{(k-1)})
      < 2k\ell\binom{n}{k-2}.
   \end{equation}
Note that one can easily get a sharper bound on $\ex_{k-1}(n,\cT)$ than \eqref{upper2}.
But \eqref{upper2} suffices for our purposes.
Suppose there exists $A\subseteq L$, where $|A|\geq t+1$, such that $\cF'_A$
 contains a copy $\cT'$ of $\cT$.
Then since each edge of $\cT'$ together with each $a\in A$ forms an edge of $\cF$,
 we can extend $\cT'$ to a copy of $\ckl$ in $\cF$,
 contradicting $\ckl\not\subseteq \cF$.
So, $\forall A\subseteq L, |A|\geq t+1$,
 we have $\cT \not \subseteq \cF_A$ and by \eqref{upper2}
$|\cF_A|\leq 2k\ell \binom{n-|L|}{k-2}$.
Let $\cF_3=\bigcup_{A\subseteq L, |A|\geq t+1} \cF_A$. By our discussion above, we
have
\begin{equation} \label{F3-upper}
|\cF_3|\leq \binom{|L|}{t+1} 2k\ell \binom{n-|L|}{k-2}<2k\ell n^{k-2+\e(t+1)}
=O(n^{k-\frac{3}{2}}).
\end{equation}

Let $\cF^*=\bigcup_{A\subseteq L, |A|\leq t} \cF_A$. Then
$\cF^*=\cF'\setminus \cF_3=\cF\setminus(\cF_1\cup \cF_2\cup \cF_3)$. By \eqref{F1-upper},
\eqref{F2-upper}, and \eqref{F3-upper}, we have
$$|\cF^*|\geq t\binom{n}{k-1}-o(n^{k-1}).$$
Furthermore, by the definition of $\cF^*$, we have
$\forall F\in \cF^*, F\cap L=\{c(F)\}$ and $\deg_{\cF^*}(F\setminus L)\leq t$.

Let $$\cF^*_0=\{F\in \cF^*: \deg_{\cF^*}(F\setminus L)\leq t-1\}.$$
Obviously
$$|\cF^*| + \frac{1}{t-1}|\cF^*_0|\leq t \binom{n-|L|}{k-1}.$$
Since $|\cF^*|\geq t\binom{n}{k-1}-o(n^{k-1})$, we  have
$$|\cF^*_0|=o(n^{k-1}) \quad \mbox{ and } |\cF^*\setminus \cF^*_0|
\geq t\binom{n}{k-1}-o(n^{k-1}).$$
By our definition, $\cF^*\setminus \cF^*_0=\bigcup_{A\subseteq L, |A|=t} \cF_A$.
Note that $|\cF_A|=t|\cF'_A|$ for each $A\subseteq L, |A|=t$.

Fix any $A\subseteq L$, $D\in \partial(\cF'_A)$, and $a\in A$. By definition, $D\cup \{a\}$
is a proper subset of some $F\in \cF_A$ where $c(F)=a$. So $D\cup \{a\}\in \Ker_s(\cF)$.
In particular, $\forall x\in \partial_1(\cF'_A)=V(\cF'_A)$ and $\forall a\in A$, we have
$\{a,x\}\in \Ker^{(2)}_s(\cF)$ and $\forall \{x,y\}\in \partial_2(\cF'_A)$ and $\forall a\in A$,
we have $\{a,x,y\}\in \Ker^{(3)}_s(\cF)$.
This means that $\Ker^{(2)}_s(\cF)$ contains all the edges between $A$ and $V(\cF'_A)$.
Let $$\cA_1=\{A\subseteq L: |A|=t, |V(\cF'_A)|\geq t+1\},  \quad  \cA_2=\{A\subseteq L:
|A|=t, |V(\cF'_A)|\leq t\}.$$
 Recall that $|L|=O(n^\varepsilon)=O(n^{\frac{1}{2(t+1)}})$. We have
$$|\bigcup_{A\in \cA_2} \cF_A|\leq \binom{|L|}{t}\binom{t}{k-1}t=O(n^\frac{1}{2}).$$
Hence,
\begin{equation} \label{A1}
|\bigcup_{A\in \cA_1} \cF_A|\geq t\binom{n}{k-1}-o(n^{k-1}).
\end{equation}

{\bf Claim 1.} $\forall A, B\in \cA_1, A\neq B$,
   we have $\partial_2(\cF'_A)\cap \partial_2(\cF'_B)=\emptyset$.

\medskip

{\it Proof of Claim 1.}
Suppose otherwise that there are $A,B\in \cA_1,  A\neq B$ and $x,y\in [n]\setminus L$,
such that $\{x,y\}\in \partial_2(\cF'_A)\cap \partial_2(\cF'_B)$.
Since $\Ker^{(2)}_s(\cF)$ contains  all the edges between $A$ and $V(\cF'_A)$ and $|V(\cF'_A)|\geq t+1$,
 we can find an $x,y$-path $P$ of length $2t$ in $\Ker^{(2)}_s(\cF)$ using the
edges between $A$ and $V(\cF'_A)$.
Let $b\in B\setminus A$. Since $\{x,y\}\in \partial_2(\cF'_B)$, we have $\{x,y,b\}\in \Ker^{(3)}_s(\cF)$.
Now, $P\cup \{x,y,b\}$ is a linear cycle of length $2t+1$ in $\Ker_s(\cF)$.
Since $s\geq k\ell$, by Proposition \ref{expansion}, $\cF$ contains a linear cycle ${\mathbb C}^{(k)}_{2t+1}$.
Note that we also have $\{x,b\}, \{y,b\}\in \Ker^{(2)}_s(\cF)$. So $P\cup \{xb,yb\}$ is
a linear cycle of length $2t+2$ in $\Ker_s(\cF)$ and by Proposition \ref{expansion}, $\cF$
contains a linear cycle of length $2t+2$. Since $\ell=2t+1$ or $2t+2$, $\cF$ contains
a copy of $\ckl$, contradicting our assumption about $\cF$. \qed

\medskip

For convenience, suppose $\cA_1=\{A_1,\ldots, A_p\}$.
For each $i\in [p]$, let $y_i\geq k-1$ denote the positive real such that
 $|\cF'_{A_i}|=\binom{y_i}{k-1}$, where without loss of generality,
 we may assume that $y_1\geq y_2\geq \ldots \geq y_p$.
By the Kruskal-Katona theorem~\eqref{eq:KK1.1},
 $\forall i\in [p], |\partial_2(\cF'_{A_i})|\geq \binom{y_i}{2}$.
By Claim 1, $\partial_2(\cF'_{A_1}),\ldots, \partial_2(\cF'_{A_p})$ are pairwise disjoint.
So we have
\begin{equation*} 
\sum_{i=1}^p\binom{y_i}{2}\leq \binom{n-|L|}{2}<\binom{n}{2}.
\end{equation*}
For each $i=1,\ldots, p$, observe that $\frac{\binom{y_i}{k-1}}{\binom{y_1}{k-1}}\leq
\frac{\binom{y_i}{2}}{\binom{y_1}{2}}$ and hence $\binom{y_i}{k-1}\leq
\binom{y_1}{k-1}\frac{\binom{y_i}{2}}{\binom{y_1}{2}}$.
This yields
\begin{equation*} 
|\bigcup_{A\in \cA_1}\cF_A|=t \sum_{A\in \cA_1}|\cF'_A|=t\sum_{i=1}^p \binom{y_i}{k-1}
\leq t\binom{y_1}{k-1}\frac{\sum_{i=1}^p \binom{y_i}{2}}{\binom{y_1}{2}}<
t\binom{y_1}{k-1}\frac{\binom{n}{2}}{\binom{y_1}{2}}.
\end{equation*}
This and  \eqref{A1}
 imply $y_1\geq n-o(n)$. Applying Kruskal-Katona theorem~\eqref{eq:KK1.1} again we get
  $$|V(\cF'_{A_1})|=|\partial_1(\cF'_{A_1})|\geq y_1\geq n-o(n).$$
Since $\Ker^{(2)}_s(\cF)$
has the all the edges between $A_1$ and $V(\cF'_{A_1})$
 the sets $S=A_1$ and $T=V(\cF'_{A_1})$ satisfy the claim of Theorem~\ref{L-set}.
\qed


\section{Proofs of the main results}\label{s:10}

In this section we prove Theorem \ref{main} and Theorem \ref{k=4}.
The lower bound is presented in Section~\ref{s:6}.
It remains to prove the upper bounds for large $n$.

Let $\cF\subseteq\binom{[n]}{k}$, where $n$ is sufficiently large.
To prove Theorem \ref{main} we assume that $k\geq 5$ and $\cF$ contains
no copy of $\ckl$. To prove Theorem \ref{k=4}, we assume that $k\geq 4$
and $\cF$ contains no member of ${\cal C}^{(k)}_\ell$.
Each upper bound in Theorem \ref{main} and \ref{k=4} is at
least $t\binom{n}{k-1}-O(n^{k-2})$. So
we may assume that $|\cF|\geq t\binom{n}{k-1}-o(n^{k-1})$.
By Theorem \ref{partition2}, we can partition $\cF$ into two subfamilies $\cF_1$ and $\cF_0$,
 where $\cF_1$ is centralized with threshold $s=k\ell$ and $|\cF_0|=O(n^{k-2})$.
In particular, $|\cF_1|\geq t\binom{n}{k-1}-o(n^{k-1})$.

By Theorem  \ref{L-set}, there exists a set $S\subseteq [n]$, where
$|S|=t$ and a set $T\subseteq [n]\setminus S$ where $|T|\geq n-o(n)$ such that
 $\Ker^{(2)}_s(\cF_1)$, as a $2$-graph,  contains all the edges between $S$ and $T$.
Let $W\subseteq [n]\setminus S$ be a set of maximum size such that
 $\Ker^{(2)}_s(\cF)$ contains all the edges between $S$ and $W$.
We have $|W|\geq n-o(n)$.
Let $Z=[n]\setminus (S\cup W)$,  $z=|Z|$.
We have $z=o(n)$.
Let $$\cF_S=\{F\in \binom{[n]}{k}: F\cap S\neq \emptyset\}.$$
Then $|\cF_S|=\binom{n}{k}-\binom{n-t}{k}$.

We split $\cF\setminus \cF_S$ into three (later into four) parts and will give an estimate
 for their sizes one by one.
We also estimate a class of missing edges, $\cD\subseteq \cF_S\setminus \cF$,
 and finally compare $|\cD|$  to $|\cF\setminus \cF_S|$.
Define $\cF\setminus \cF_S=\cG_0\cup \cG_1 \cup \cG_2$, $\cG_1=\cA\cup \cB$ and $\cD$
 as follows.
\begin{eqnarray*}
\cG_0&=&\{F\in \cF: F\subseteq Z\}, \quad \text{i.e., }F\cap S=\emptyset, \, |F\cap W|=0,\\
\cG_1&=&\{F\in \cF: F\cap S=\emptyset,\, |F\cap W|= 1\},\\
 &\cA&=\{F\in \cG_1: \deg_{\cG_1}(F\setminus W) < \ell\}, \quad
   \cB=\{F\in \cG_1: \deg_{\cG_1}(F\setminus W) \geq \ell \},\\
\cG_2&=&\{F\in \cF: F\cap S=\emptyset,\, |F\cap W|\geq 2\},\\
\cD&=&\{F\in \binom{[n]}{k}: |F\cap S|=|F\cap Z|=1, \, F\notin \cF\}.
   \end{eqnarray*}

The family $\cG_0$ does not contain a $\ckl$ on $z$ vertices so Proposition \ref{cycle-bound} yields
  \begin{equation} \label{z-bound}
|\cG_0|\leq k \ell  \binom{z}{k-1}=O(z^{k-1})=o(zn^{k-2}).
\end{equation}

Clearly
\begin{equation}\label{eq:A}
 |\cA|\leq \ell \binom{|Z|}{k-1}=O(z^{k-1})=o(zn^{k-2}).
   \end{equation}

Let $\cB'=\{F\setminus W: F\in \cB\}$,
 it is a $(k-1)$-graph on $Z$.
If $\cB'$ contains a copy $L$ of ${\mathbb C}_\ell^{(k-1)}$, then
since $\forall F\in \cB, \deg_{\cG_1}(F\setminus W)\geq \ell$,
$L$ can be extended to a copy of $\ckl$ in $\cF$, a contradiction.
So $\cB'$ contains no linear $\ell$-cycle, and Proposition \ref{cycle-bound} gives
$|\cB'|\leq k\ell \binom{z}{k-2}=O(z^{k-2})$.
Since $|\cB|\leq |\cB'|\cdot|W|$ we get
\begin{equation} \label{B-bound}
   |\cB|=O(z^{k-2}n)=o(z n^{k-2}).
\end{equation}

\medskip\noindent
{\bf Claim 2.} For $\ell=2t+1$ we have $\cG_2=\emptyset$.
For $\ell=2t+2$, if $\cF$ has no linear $\ell$-cycle then $\cG_2$ has no two members
 meeting in a singleton and if $\cF$ has no minimal $\ell$-cycle then $|\cG_2|\leq 1$.

\medskip

\noindent
{\it Proof of Claim 2.}\enskip
Suppose first that $\ell=2t+1$.
Suppose $\cG_2$ has a member $F$.
By definition, $F\cap S=\emptyset$ and $|F\cap W|\geq 2$.
Let $x,y$ be two elements of $F\cap W$.  Let $H$ denote the subgraph of $\Ker^{(2)}_s(\cF)$
 consisting of  all of its edges between $S$ and $W$.
By our choice of $W$, $H$ is a complete bipartite graph.
Since $|W|\geq n-o(n)>t+k$, for large $n$, we can find an $x,y$-path $P$ of length $2t$
 in $H$ such that $P\cap F=\{x,y\}$.
Since each edge on $P$ has kernel degree at least
 $s=k\ell$ in $\cF$, we can expand $F\cup P$ into a linear $\ell$-cycle in $\cF$, a contradiction.
So $\cG_2=\emptyset$.

Next, consider the case $\ell=2t+2$.
Suppose that $\cF$ has no linear $\ell$-cycle and $\cG_2$ contains two members $F$ and $F'$
that intersect in exactly  one element $u$.
Let $x$ be a vertex in $(F\cap W)\setminus \{u\}$ and $y$ a vertex in $(F'\cap W)\setminus\{u\}$.
Like before, since $H$ has all the edges between $S$ and $W$ and $|W|$ is large,
 we can find an $x,y$-path in $H$ of length $2t$ such that $P\cap (F\cup F')=\{x,y\}$.
We can expand $F\cup F'\cup P$ into a linear cycle of length $2t+2=\ell$ in $\cF$, a contradiction.
Suppose $\cF$ contains no minimal $\ell$-cycle instead and $\cG_2$ contains two different edges
 $F$ and $F'$.
Then we can get a contradiction by constructing a minimal $\ell$-cycle in $\cF$ using a procedure
 similar to above.
We omit the details.
\qed

\medskip

For $\ell=2t+1$, by Claim 2 we have $|\cG_2|=0$.
For $\ell=2t+2$, if $\cF$ has no linear $\ell$-cycle, then Claim 2 and
Frankl's theorem~\eqref{eq:EKR} yield
$|\cG_2|\leq \binom{n-t-2}{k-2}$ (for large enough $n$)
and if $\cF$ has no minimal $\ell$-cycle  then $|\cG_2|\leq 1$.

Finally, consider $\cD$.
Let $u\in Z$.
The maximality of $W$ implies that there exists an $x\in S$
 such that $xu\notin \Ker^{(2)}_s(\cF)$ and hence $\deg^*_{\cF}(\{x,u\})<s$.
 Then~\eqref{eq:2.3} implies that the
  $\deg_\cF( \{ x,u\})\leq s\binom{|W|-1}{k-3}$.
Hence $\deg_\cD(\{x,u\})\geq \binom{|W|}{k-2}-s\binom{|W|-1}{k-3}\geq \Omega(n^{k-2})$.
Since this holds for every $u\in Z$  we get
\begin{equation} \label{D-bound}
|\cD|\geq |Z|\times \left( \binom{|W|}{k-2}-s\binom{|W|-1}{k-3}\right)\geq \Omega(z n^{k-2}).
\end{equation}

Now, we are ready to prove the desired bound on $|\cF|$.
By our definition, $\cF\subseteq (\cF_S\setminus \cD)\cup G_0\cup \cG_1\cup \cG_2$.
For $\ell=2t+1$ we have $|\cG_2|=0$. So, for sufficiently large $n$, \eqref{z-bound}--\eqref{D-bound} yield
\begin{equation} \label{odd-upper}
|\cF|\leq |\cF_S|-\Omega(zn^{k-2})+o(zn^{k-2})\leq |\cF_S|-
\Omega(zn^{k-2}) \leq |\cF_S|=\binom{n}{k}-\binom{n-t}{k}.
\end{equation}

For $\ell=2t+2$, if we assume  that $k\geq 5$ and $\cF$ has no linear $\ell$-cycle, then
$|\cG_2|\leq \binom{n-t-2}{k-2}$ and by \eqref{z-bound}--\eqref{D-bound} we have
\begin{equation} \label{even-upper}
|\cF|\leq |\cF_S|-\Omega(zn^{k-2})+o(zn^{k-2})
  +\binom{n-t-2}{k-2}\leq \binom{n}{k}-\binom{n-t}{k}+\binom{n-t-2}{k-2},
  \end{equation}
for large $n$.
If we assume that $k\geq 4$ and $\cF$ has no minimal $\ell$-cycle, then $|\cG_2|\leq 1$
 and we have $|\cF|\leq |\cF_S|+1=\binom{n}{k}-\binom{n-t}{k}+1$.
\qed

\section{Stability and concluding remarks}

By \eqref{odd-upper} and \eqref{even-upper}, we also have the following
stability statement.

\begin{proposition}
Let $k,\ell$ be positive integers, where $\ell\geq 3$ and $k\geq 4$.
Let $\varepsilon$ be any small positive real.
There exists a positive real $\delta$ such that for all  $n\geq n_2(k,\ell)$ the follows holds.
Let $\cF\subseteq \binom{[n]}{k}$ be a family that contains no copy of $\ckl$ if $k\geq 5$
 and no member of $\cckl$ if $k=4$ and $|\cF|\geq (1-\delta) t \binom{n}{k-1}$.
Then there exists a set $S\subseteq [n]$, where $|S|=t$, such that all except at most
 $\varepsilon \binom{n}{k}$ of the members of $\cF$ intersect $S$.
\end{proposition}

In Section~\ref{s:3} we observed that
the $3$-uniform linear cycle $\cthreel$ is a subgraph of the triangulated cycle
$\TT^{(3)}_\ell$ and $\ckl$ is a $k$-expansion of $\cthreel$.
The triangulated cycle $\TT^{(k)}_\ell$ is an example of a so-called  $q$-forest
 where $q=3$.
A {\it $q$-forest} is a $q$-graph whose edges can be linearly ordered as
 $E_1,\ldots, E_m$ such that for all $i\geq 2$ there exists
some $a(i)<i$ such that $E_i\cap ((\bigcup_{j<i} E_j)\subseteq E_{a(i)}$. A subgraph of a $q$-forest
is called a {\it partial $q$-forest}. So $\cthreel$ is a partial $3$-forest.
In a forthcoming paper, for all $k,q$ satisfying $q\geq 3$ and $k\geq 2q-1$,
 we will asymptotically determine the Tur\'an numbers for the rather wide family of hypergraphs
 that are $k$-expansions of partial $q$-forests.



\end{document}